\documentclass[12pt]{article}
\usepackage[english]{babel}
\usepackage[all,cmtip]{xy}
\usepackage{amssymb,amsmath,mathrsfs,amsthm}
\usepackage{color}

\newcommand{\defi}[1]{\textsf{#1}} 

\newcommand\wt[1]{\widetilde{#1}}

\newtheorem{thm}{Theorem}

\newtheorem{defn}[thm]{Definition}
\newtheorem{cor}[thm]{Corollary}

\newtheorem{lem}[thm]{Lemma}
\newtheorem{prop}[thm]{Proposition}

\newtheorem{ex}[thm]{Example}

\newcommand{\p}{\mathbf{P}}

\newcommand{\Q}{\mathbf{Q}}

\newcommand{\Z}{\mathbf{Z}}
\newcommand{\m}{\mathfrak{m}}
\newcommand{\OO}{\mathcal{O}}
\newcommand{\DD}{\mathcal{D}}

\newcommand{\F}{\mathbf{F}}

\newcommand{\EE}{E}

\newcommand{\mm}{\mathfrak{m}}

\newcommand{\Spec}{\operatorname{Spec}}

\newcommand{\wht}{\operatorname{wt}}

\newcommand{\sm}{\operatorname{sm}}

\newcommand{\PGL}{\textnormal{PGL}}

\begin{document}

\title{On the structure of elliptic curves over finite extensions of $\Q_p$ with additive reduction}
\author{Michiel Kosters and Ren\'e Pannekoek}
\date{\today}

\maketitle

\begin{abstract} \noindent Let $p$ be a prime and let $K$ be a finite extension of $\Q_p$. Let $E/K$ be an elliptic curve with additive reduction. In this paper, we study the topological group structure of the set of points of good reduction of $E(K)$. In particular, if $K/\Q_p$ is unramified, we show how one can read off the topological group structure from the Weierstrass coefficients defining $E$.
\end{abstract}

\section{Introduction}

\noindent
In this article, we fix a prime $p$. Let $K/\Q_p$ be a finite extension. Let $\mathcal{O}_K$ be the ring of integers of $K$ with maximal ideal $\m_K$ and residue field $k$. Let $E/K$ be an elliptic curve 
with additive reduction, given by a Weierstrass equation over $\mathcal{O}_K$ of the form
$$
Y^2 + a_1 XY + a_3 Y = X^3 + a_2 X^2 + a_4 X + a_6,~~~~\textnormal{$a_i \in \m_K$ for each $i$}.
$$
We denote by $E_0(K) \subset E(K)$ the subgroup of points that reduce to a non-singular point of the reduction curve defined over $k$. 

The purpose of this paper is to investigate the structure of $E_0(K)$ as a topological group and as a $\Z_p$-module. We first show that $p$-adic topology on $E_0(K)$ from the embedding into $\p^2(K)$ agrees with the topology from the $\Z_p$-module structure (Proposition \ref{888} and Proposition \ref{fund_isom}). Our main theorem is the following, where $N_{k/\F_p}: k \to \F_p$ is the norm map.

\begin{thm} \label{777}
\label{main}
Assume that $K/\Q_p$ is unramified of degree $n$ with ring of integers $\mathcal{O}_K$, maximal ideal $\m_K$ and residue field $k$. Let $E/K$ be an elliptic curve given by a Weierstrass equation over $\mathcal{O}_K$ of the form
$$
Y^2 + a_1 XY + a_3 Y = X^3 + a_2 X^2 + a_4 X + a_6,
$$
with $a_i \in \m_K=p \mathcal{O}_K$ for each $i$. Then one has $E_0(K) \cong_{\Z_p} \Z_p^n$, except in the following four cases:
\begin{itemize}\itemsep=0pt
\item[(i)] $p=2$ and the equation $\overline{a_3/2}x^3+\overline{a_1/2}x-1=0$ has a solution in $k$;
\item[(ii)] $p=3$ and $N_{k/\F_p}(\overline{8a_2/3})=1$;
\item[(iii)] $p=5$ and $N_{k/\F_p}(\overline{3a_4/5})=1$;
\item[(iv)] $p=7$ and $N_{k/\F_p}(\overline{4a_6/7})=1$.
\end{itemize}
In case (i), one $E_0(\Q_2) \cong_{\Z_2} \Z_2^n \times (\Z/2\Z)^b$ where $2^b$ is the number of solutions to $\overline{a_3/2}X^4+\overline{a_1/2}X^2-X=0$ in $k$. In the cases (ii)-(iv), one has $E_0(\Q_p) \cong_{\Z_p} \Z_p^n \times \Z/p\Z$. 
\end{thm} 

The above theorem gives the following result for $K=\Q_p$. 

\begin{cor} \label{555}
Let $E/\Q_p$ be an elliptic curve given by a Weierstrass equation over $\Z_p$
$$
Y^2 + a_1 XY + a_3 Y = X^3 + a_2 X^2 + a_4 X + a_6,
$$
with $a_i \in p \Z_p$ for each $i$. 
One has $E_0(\Q_p) \cong_{\Z_p} \Z_p$, unless one is in one of the four special cases:
\begin{itemize}\itemsep=0pt
\item[(i)] $p=2$ and $a_1+a_3 \equiv 2 \pmod{4}$;
\item[(ii)] $p=3$ and $a_2 \equiv 6 \pmod{9}$;
\item[(iii)] $p=5$ and $a_4 \equiv 10 \pmod{25}$;
\item[(iv)] $p=7$ and $a_6 \equiv 14 \pmod{49}$.
\end{itemize}
In all special cases one has $E_0(\Q_p) \cong_{\Z_p} \Z_p \times \Z/p\Z$.
\end{cor}

Note that any curve with additive reduction can be written in the form of Theorem \ref{777} (Lemma \ref{999}).
The case $p>7$ of Theorem \ref{main} was also mentioned in \cite[Lemma 1]{SWID}. 

We obtain the following result when $K/\Q_p$ is ramified.

\begin{thm}
Assume that $K/\Q_p$ is of degree $n$ and that the ramification index is $e$. If $6e < p-1$, then one has $\EE_0(K) \cong_{\Z_p} \Z_p^n$. 
\end{thm}
We will briefly discuss some of the exceptional cases for ramified extensions, but we do not obtain a completely satisfying answer.

We will say a few words about the idea of the proofs. Let $K/\Q_p$ be a finite extension. It is a standard fact from the theory of elliptic curves over local fields \cite[VII.6.3]{SI} that $E_0(K)$ admits a canonical filtration
$$
E_0(K) \supset E_1(K) \supset E_2(K) \supset E_3(K) \supset \ldots,
$$
where for each $i \geq 1$ the quotient $E_i(K)/E_{i+1}(K)$ is isomorphic to $k$. The quotient $E_0(K)/E_{1}(K)$ is also isomorphic to $k$ by the fact that $E$ has additive reduction. The groups $E_i(K)$ for $i \geq 1$ can be seen as formal groups, and since our curve has additive reduction, $E_0(K)$ can also be seen as a formal group (Proposition \ref{tocomputewith}). For large enough $j$, the theory of formal groups gives an isomorphism $j : E_j(K) \stackrel{\sim}{\rightarrow} \m_K^j$. The idea is then to explicitly use the exact sequence $0 \to E_j(K) \to E_{j-1}(K) \to k \to 0$ to compute the structure of $E_{j-1}(K)$. We keep doing this, until we (hopefully) get the structure of $E_0(K)$.  In the case that $K$ is unramfied over $\Q_p$, the computations become easier, since one already has $E_1(K) \cong \m_K$ (Proposition \ref{144}) and we have to study only one exact sequence to obtain the structure of $E_0(K)$.

This paper is a generalization of \cite{PAN} and hence contains some overlap.

\section{Preliminaries}
\label{weier}

\subsection{Weierstrass curves}

All proofs of facts recalled in this section can be found in \cite[IV and VII]{SI}.

Let $p$ be a prime. Let $K$ be a finite field extension of $\Q_p$ of degree $n$ and ramification degree $e$. Let $v_K : K \rightarrow \Z \cup \{ \infty \}$ be its normalized valuation. Let $\OO_K$ be the ring of integers, $\mathfrak{m}_K$ its maximal ideal and $k$ its residue field. By a \defi{Weierstrass curve} over $\OO_K$ we mean a projective curve $\EE \subset \p^2_{\OO_K}$ defined by a Weierstrass equation
\begin{equation}
\label{Weiereq}
Y^2 + a_1XY + a_3Y = X^3+ a_2X^2 + a_4X + a_6,
\end{equation}
such that the generic fiber $\EE_K$ is an elliptic curve with $(0:1:0)$ as the origin. The coefficients $a_i \in \OO_K$ are uniquely determined by $\EE$.

The set $\p^2(K)$ has the quotient topology from $K^3$ where $K$ has the $p$-adic topology. This induces a topology on $\EE(K)$ and makes $\EE(K)$ into a topological group. We call this topology the standard topology. 

We will say that a Weierstrass curve $\EE/\OO_K$ has \defi{good reduction} when the special fiber $\wt{\EE}=\EE_{k}$ is smooth, \defi{multiplicative reduction} when $\wt{\EE}$ is nodal (i.e. there are two distinct tangent directions to the singular point), and \defi{additive reduction} when $\wt{\EE}$ is cuspidal (i.e. one tangent direction to the singular point).

We have $\EE(K) = \EE(\OO_K)$ since $\EE$ is projective. Therefore, we have a reduction map $\EE(K) \rightarrow \wt{\EE}(k)$ given by restricting an element of $\EE(\OO_K)$ to the special fiber. Let $\wt{\EE}_{\textnormal{sm}}(k)$ be the complement of the singular points in the special fiber $\wt{\EE}(k)$, which is a group.  We denote the preimage of $\wt{\EE}_{\textnormal{sm}}(k)$ under the reduction map by $\EE_0(K)$. The standard topology makes $\EE_0(K)$ into a topological group. By $\EE_1(K) \subset \EE_0(K)$ we denote the \defi{kernel of reduction}, i.e. the points that map to the identity $\infty$ of $\wt{\EE}(k)$. A more explicit definition of $\EE_1(K)$ is
\begin{equation}
\label{defE1}
\EE_1(K) = \left\{ (x,y) \in \EE(K) : v_K(x) \leq -2, v_K(y) \leq -3 \right\} \cup \{ \infty \}.
\end{equation}
More generally, one defines subgroups $\EE_n(K) \subset \EE_0(K)$ as follows \cite[ex.~7.4]{SI}:
$$
\EE_n(K) = \left\{ (x,y) \in \EE(K) : v_K(x) \leq -2n, v_K(y) \leq -3n \right\} \cup \{ \infty \}.
$$
We thus have an infinite filtration on the subgroup $\EE_1(K)$:
\begin{equation}
\label{filtE}
\EE_1(K) \supset \EE_2(K) \supset \EE_3(K) \supset \cdots
\end{equation}
This gives us a filtration of open neighborhoods around $\infty$ of $\EE_0(K)$. 

We have an exact sequence \cite[VII.2.1]{SI}:
\begin{equation}
0 \rightarrow \EE_1(K) \rightarrow \EE_0(K) \rightarrow \wt{\EE}_{\textnormal{sm}}(k) \rightarrow 0.
\end{equation}

If $\wt{\EE}$ is smooth, $\wt{\EE}=\wt{\EE}_{\textnormal{sm}}(k)$ is an elliptic curve.
If $\wt{\EE}$ is cuspidal, one has $\wt{\EE}_{\textnormal{sm}}(k) \cong k^*$ or  $\wt{\EE}_{\textnormal{sm}}(k) \cong \{c \in l^*: N_{l/k}(c)=1\}$ where $l/k$ is of degree $2$. We call this multiplicative reduction. If $\EE_{k}$ is nodal, one has $\wt{\EE}_{\textnormal{sm}}(k) \cong k^+$. We call this additive reduction. Assume that $\wt{\EE}$ has additive reduction. If $(0,0)$ is the node, and the tangent line is given by $Y=0$, that is, the equation for the reduction is of the form $Y^2=X^3$, then one has an isomorphism $\wt{\EE}_{\textnormal{sm}}(k) \to k^+$ is given by $(x,y) \mapsto -x/y$. See \cite[III.2.5]{SI}. In the case that the curve is minimal and of additive reduction, the quotient $E(K)/\EE_1(K)$ has order at most $4$ \cite[VII.6.1]{SI}.

We denote the \defi{formal group} corresponding to $E$ by $\widehat{\EE}$ \cite[IV.1--2]{SI}. This is a one-dimensional formal group over $\OO_K$. Giving the data of this formal group is the same as giving a power series $F = F_{\widehat{\EE}}$ in $\OO_K[[X,Y]]$, called the \defi{formal group law}. It satisfies
$$
F(X,Y) = X + Y + (\textnormal{terms of degree $\geq 2$})
$$
and
$$F(F(X,Y),Z)) = F(X,F(Y,Z)).
$$
For $\EE$ as in (\ref{Weiereq}), the first few terms of $F$ are given by:
\begin{align*}
& F(X,Y) \,= \\
& X+Y \, - \, a_1XY \, - \, a_2(X^2Y+XY^2)  \,-\,2a_3(X^3Y + XY^3) + (a_1a_2 - 3a_3)X^2Y^2 \,-\\
& (2a_1a_3+2a_4)(X^4Y+XY^4)-(a_1a_3-a_2^2+4a_4)(X^3Y^2+X^2Y^3) \,+ \ldots.
\end{align*}
In fact, $\widehat{\EE}$ is a $\Z_p$-module (see \cite[ex.~4.3]{SI}).

Treating the Weierstrass coefficients $a_i$ as unknowns, we may consider $F$ as an element of $\Z[a_1,a_2,a_3,a_4,a_6][[X,Y]]$, called the \defi{generic formal group law}. If we make $\Z[a_1,a_2,a_3,a_4,a_6]$ into a weighted ring with weight function wt, such that $\wht(a_i)=i$ for each $i$, then the coefficients of $F$ in degree $n$ are homogeneous of weight $n-1$ \cite[IV.1.1]{SI}. For each $n \in \Z_{\geq 2}$, we define power series $[n]$ in $\OO_K[[T]]$ by $[2](T) = F(T,T)$ and $[n](T) = F([n-1](T),T)$ for $n\geq 3$. Here also, we may consider each $[n]$ either as a power series in $\OO_K[[T]]$ or as a power series in $\Z[a_1,a_2,a_3,a_4,a_6][[T]]$ called the \defi{generic multiplication by $n$ law}. We have:
\begin{lem}
\label{mult_by_p}
Let $[p] = \sum_i b_i T^i \in \Z[a_1,a_2,a_3,a_4,a_6][[T]]$ be the generic formal multiplication by $p$ law. Then:
\begin{enumerate}\itemsep=0pt
\item \label{p_deelt_bn} $p \mid b_i$ for all $i$ not divisible by $p$;
\item \label{gewicht_bn} $\textnormal{wt}(b_i) = i-1$, considering $\Z[a_1,a_2,a_3,a_4,a_6]$ as a weighted ring as above.\end{enumerate}
\end{lem}
\begin{proof}
(\ref{p_deelt_bn}) is proved in \cite[IV.4.4]{SI}. \\
(\ref{gewicht_bn}) follows from \cite[IV.1.1]{SI} or what was said above.
\end{proof}
The series $F(u,v)$ converges to an element of $\mm_K$ for all $u,v \in \mm_K$. To $\EE$ one associates the group $\hat\EE(\mathfrak{m}_K)$, the $\mathfrak{m}_K$-valued points of $\widehat{\EE}$, which as a set is just $\mathfrak{m}_K$, and whose group operation $+$ is given by $u+v = F(u,v)$ for all $u,v \in \hat\EE(\mathfrak{m}_K)$. The identity element of $\hat\EE(\mathfrak{m}_K)$  is $0 \in \mm_K$. If $i \geq 1$ is an integer, then by $\hat\EE(\mathfrak{m}^i_K)$ we denote the subset of $\hat\EE(\mathfrak{m}_K)$ corresponding to the subset $\mm_K^i \subset \mm_K$, where $\mm_K^i$ is the $i$th power of the ideal $\mm_K$ of $\OO_K$. The subsets $\hat\EE(\mathfrak{m}^i_K)$ are subgroups of $\hat\EE(\mathfrak{m}_K)$, and we have an infinite filtration of $\widehat{\EE}(\mm_K)$:
\begin{equation}
\label{filtEhat}
\widehat{\EE}(\mm_K) \supset \widehat{\EE}(\mm_K^2) \supset \widehat{\EE}(\mm_K^3) \supset \cdots.
\end{equation}
One has $\widehat{\EE}(\mm_K^i)/\widehat{\EE}(\mm_K^{i+1}) \cong \m_K^i/\m_K^{i+1} \cong k^+$. For $i>e/(p-1)$ one has $\widehat{\EE}(\mm_K^i)  \cong_{\Z_p} \m_K^i $ by a `logarithm' map, and the following diagram commutes (\cite[IV.6.4]{SI}):
\begin{align*}
\xymatrix{
\widehat{\EE}(\mm_K^{i+1} ) \ar@{^{(}->}[r] \ar[d]^{\sim} & \widehat{\EE}(\mm_K^i) \ar[d]^{\sim}  \\
\m_K^{i+1} \ar@{^{(}->}[r] & \mm_K^{i}.
}
\end{align*}

We make $\widehat{\EE}(\mm_K) $ into a topological group by using the above filtration as a fundamental system of neighborhoods around $0$. This means that a subset $U$ of $\widehat{\EE}(\mm_K)$ is open if and only if for all $x \in U$ there is an $m \in \Z_{\geq 1}$ such that $x+\widehat{\EE}(\mm_K^m) \subseteq U$. This is the same as the topology coming from the identification $\widehat{\EE}(\mm_K)=\m_K$ of sets and then using the $p$-adic topology from $K$ on $\m_K$. We call this the standard topology.

\begin{prop}\label{fund_isom}
The map
\begin{align*}
\psi_K: \EE_1(K) & \stackrel{\sim}{\rightarrow} \widehat{\EE}(\mathfrak{m}_K) \\
(x,y) & \mapsto -x/y \\
0 & \mapsto 0
\end{align*}
is an isomorphism of topological groups. Moreover, $\psi_K$ respects the filtrations (\ref{filtE}) and (\ref{filtEhat}), i.e. it identifies the subgroups $\EE_i(K)$ defined with $\hat\EE(\mathfrak{m}^i_K)$ for $i \in \Z_{\geq 1}$.
\end{prop}
\begin{proof}
See the proof of \cite[VII.2.2]{SI}: the maps are homeomorphisms.
\end{proof}
Given a finite field extension $L \supseteq K$, we have a natural commutative diagram
\[
\xymatrix{
\EE_1(K) \ar[r]^{\psi_K}   \ar@{^(->}[d]^{\textnormal{incl}} & \widehat{\EE}(\mathfrak{m}_K) \ar@{^(->}[d]^{\textnormal{incl}}   \\\
\EE_1(L) \ar[r]^{\psi_L}                            & \widehat{\EE_{\OO_L}}(\mathfrak{m}_L)       
}
\]
Here $\widehat{\EE_{\OO_L}}(\mathfrak{m}_L)$ is the set of $\mathfrak{m}_L$-valued points of the formal group of $\EE_{\OO_L}$, the base-change of $\EE$ to $\Spec(\OO_L)$.

Let $G$ be a finitely generated $\Z_p$-module. We make $G$ into a topological group by using the filtration $\{p^i G: i \in \Z_{\geq 0}\}$ of neighborhoods around $0$. We call this the $\Z_p$-topology.


\begin{prop} \label{888}
The group $\EE_0(K)$ is a finitely generated $\Z_p$-module. Furthermore, the $\Z_p$-module topology and the standard topology on $\EE_0(K)$ coincide.
\end{prop}
\begin{proof} We use Proposition \ref{fund_isom}. Note that $\EE_0(K)/\EE_1(K)$ is finite. One has $\EE_i(K)/\EE_{i+1}(K) \cong k^+$ for $i \geq 1$, and for large enough $i$ one has $\EE_i(K) \cong_{\Z_p} \m_K^i$. Note that $\m_K^i$ is finitely generated and hence $\EE_i(K)$ is finitely generated. This shows that $\EE_0(K)$ is finitely generated.

For large enough $i$ one has $p \EE_i(K) \cong p \m_K^i \cong \m_K^{i+e} \cong \EE_{i+e}(K)$. Hence the filtrations of neighborhoods give the same topology. 
\end{proof}

\subsection{Commutative algebra}

\begin{lem} \label{128} Let $a, n \in \Z_{\geq 0}$. Let 
\begin{align*}
0 \to \Z_p^a \overset{i}{\to} G \overset{\pi}{\to} H \to 0
\end{align*}
be an exact sequence of $\Z_p$-modules where $H$ is finitely generated and $p^n$-torsion. Let $\tau: \Z_p^a \to \Z_p^a/p^n \Z_p^a$ be the natural quotient map.
Define 
\begin{align*}
g: H \to& \Z_p^a/p^n \Z_p^a \\
\pi(x) \mapsto& \tau ( p^nx).
\end{align*}
Then $g$ is a well-defined morphism of $\Z_p$-modules and one has $G \cong \Z_p^a \oplus \mathrm{ker}(g)$. Furthermore, if $\mathrm{ker}(g)=0$, one can identify $G$ with $1/p^n(\tau^{-1} \mathrm{im}(g)) \supseteq \Z_p^a$ inside $\Q_p^a$.  
\end{lem} 
\begin{proof}
Consider the morphism $h: G \to \Z_p^a/p^n \Z_p^a$, $x \mapsto \tau(p^nx)$. Note that $h|_H=0$, so this morphism induces the map $g$ with $h= g \circ \pi$. By the structure theorem for finitely generated $\Z_p$-modules we have $G \cong \Z_p^a \oplus G[p^n]$. We claim $G[p^n] + \Z_p^a = \ker(h)$. Note that $G[p^n] + \Z_p^a \subseteq \ker(h)$. Let $x \in  \ker(h)$. Then one has $p^nx \in p^n \Z_p^a$, equivalently $p^n(x-c)=0$ for some $c \in \Z_p^a$. We conclude $x \in G[p^n]+\Z_p^a$. Since $\pi$ is surjective, and trivial on $\Z_p^a$, and injective on $G[p^n]$, the kernel of $g$ is isomorphic to $G[p^n]$.

Assume that $g$ is injective. One has $\tau^{-1} \mathrm{im}(g)=p^n G$, and from this the final result follows easily. 

\end{proof}

\begin{lem} \label{183}
Let $k$ be a finite field of characteristic $p$. Consider the additive polynomial $f=X-a X^p \in k[x]$ with $a \in k^*$. Then $f$ has all its roots in $k$ is and only if $a$ is a $(p-1)$st power, that is, if and only if $N_{k/\F_p}(a)=a^{(\# k-1)/(p-1)}=1$. 
\end{lem}
\begin{proof}
The roots in $\overline{k}$ form a subgroup form a group isomorphic to $\Z/p\Z$. The polynomial $f$ has a non-trivial root in $k$ if and only if $1/a=X^{p-1}$ has a solution, that is, if and only if $a$ is a $(p-1)$st power.
\end{proof}

\section{Weierstrass curves with additive reduction over $\OO_K$} \label{112}

As in Section \ref{weier}, let $K$ be a finite extension of $\Q_p$ of degree $n$. Let $\OO_K$ be the ring of integers of $K$, with maximal ideal $\mathfrak{m}_K$ and residue field $k$ and ramification degree $e=e(K/\Q_p)$.

In this section, we study Weierstrass curves over $\OO_K$ with additive reduction. 

\begin{lem} \label{999}
\label{ai_in_maxideaal}
Let $\EE/\OO_K$ be a Weierstrass curve with additive reduction. Then $\EE$ is $\OO_K$-isomorphic to a Weierstrass curve of the form
$$
Y^2 + a_1 XY + a_3 Y = X^3 + a_2 X^2 + a_4 X + a_6,
$$
where all $a_i$ lie in $\mathfrak{m}_K$.
\end{lem}
\begin{proof}
We construct an automorphism $\alpha \in \PGL_3(\OO_K)$ that maps $\EE$ to a Weierstrass curve of the desired form. Consider a translation $\alpha_1 \in \PGL_3(\OO_K)$ moving the singular point of the special fiber $\EE_k$ to $(0:0:1)$. The image $\EE_1 = \alpha_1(\EE)$ is a Weierstrass curve with coefficients satisfying $a_3,a_4,a_6$ in $\mm_K$. There exists a second automorphism $\alpha_2 \in \PGL_3(\OO_K)$, of the form $X' = X, Y' = Y + cX$, such that in the special fiber of $\alpha_2(\EE_1)$ the unique tangent at $(0:0:1)$ is given by $Y'=0$. The Weierstrass curve $\EE_2=\alpha_2(\EE_1)$ now has all its coefficients $a_1,a_2,a_3,a_4,a_6$ in $\mm_K$. One may thus take $\alpha = \alpha_2 \circ \alpha_1$.
\end{proof}

From now on we assume that $\EE/\OO_K$ is a Weierstrass curve given by (\ref{Weiereq}), and we suppose that the $a_i$ are contained in $\mm_K$. In particular, $\EE$ has additive reduction. 

\subsection{$\EE_i(K)$, $i>0$}

\begin{prop} \label{144}
Let $\EE/\OO_K$ be a Weierstrass curve given by (\ref{Weiereq}), and assume that the $a_i$ are contained in $\mm_K$. 
 For $i>e/(p-1)$ or if $p=2$ and $i \geq e/(p-1)$, one has $\EE_i (K)\cong_{\Z_p} \Z_p^n$ and $p\EE_{i}(K)=\EE_{i+e}(K)$. 
\end{prop}
\begin{proof}
One has $\EE_i (K) \cong \widehat{\EE}(\mm_K^i)$ as we have seen before. If $i>e/(p-1)$, the result we want to prove is precisely \cite[IV.6.4]{SI}. 

Assume $p=2$ and $i>e/(p-1)-1=e-1$. Note $i  \geq 1$ and one has $i+1>e/(p-1)$. We have the exact sequence
\begin{align*}
0 \to \widehat{\EE}(\mm_K^{i+1}) \to \widehat{\EE}(\mm_K^i) \to \m^i/\m^{i+1} \cong k \to 0.
\end{align*}
One has $2 \widehat{\EE}(\mm_K^{i+1})=\widehat{\EE}(\mm_K^{i+1+e})$. Set $[2](T)=\sum_j b_j T^j$. By Lemma \ref{128} we have an induced map 
\begin{align*}
g: \m_K^i/\m_K^{i+1} \to& \widehat{\EE}(\mm_K^{i+1})/\widehat{\EE}(\mm_K^{i+1+e}) \\
a \pmod{\m_K^{i+1}} \mapsto& [2](a)= \sum_j b_j a^j \pmod{\widehat{\EE}(\mm_K^{i+1+e})}.
\end{align*}
We then use Lemma \ref{mult_by_p}. If $j>1$ and $2 \nmid j$, one has $v(b_j) \geq e+1$ and hence $v(b_j a^j) \geq v(b_j)+v(a) \geq e+i+1$. If $j=2j' \geq 2$, then one has $v(b_j)  \geq 1=p-1$ (here we require $p=2$) and $v( b_j a^j) \geq 1 + 2i > e+i$. As $b_0=2$, the induced map is just $g: a \mapsto 2a$, which is injective. The image of $g$ is $\widehat{\EE}(\mm_K^{i+e})/\widehat{\EE}(\mm_K^{i+1+e})$.
We can identify  $\widehat{\EE}(\mm_K^{i+e})$ with $\m_K^{i+e}$ through the logarithm and by Lemma \ref{128} we identify  $\widehat{\EE}(\mm_K^{i})$  with $1/2 \m_K^{i+e}=\m_K^i$. In other words, one has $2\EE_{i}(K)=\EE_{i+e}(K)$.
\end{proof}

\subsection{$\EE_0(K)$: general theory}

If we let $F$ denote the formal group law of $\EE$, then the assumption on the $a_i$ implies that $F(u,v)$ converges to an element of $\OO_K$ for all $u,v \in \OO_K$. Hence $F$ can be seen to induce a group structure on $\OO_K$, extending the group structure on $\widehat{\EE}(\mm_K)$. The same statement holds true when we replace $K$ by a finite field extension $L$.
\begin{defn}
\upshape
Let $\EE/\OO_K$ be a Weierstrass curve given by (\ref{Weiereq}), and assume that the $a_i$ are contained in $\mm_K$. For any finite field extension $L \supseteq K$, we denote by $\widehat{\EE}(\OO_L)$ the topological group obtained by endowing the space $\OO_L$ with the group structure induced by $F$.
\end{defn}
The group $\widehat{\EE}(\OO_K)$ is a $\Z_p$-module, and comes with the $\Z_p$-module topology (equivalently, the $p$-adic topology from the set $\OO_K$).

\begin{thm} 
\label{tocomputewith}
Let $\EE/\OO_K$ be a Weierstrass curve given by (\ref{Weiereq}), and assume that the $a_i$ are contained in $\mm_K$. 
\begin{enumerate}\itemsep=0pt
\item \label{Fhomeo} The map $\Psi : \EE_0(K) \rightarrow \widehat{\EE}(\OO_K)$ that sends $(x,y)$ to $-x/y$ is an isomorphism of topological groups.
\item \label{pgt7} If $6e < p-1$, then one has $\EE_0(K) \cong_{\Z_p} \Z_p^n$. 
\item One has the following commutative diagram, where $\EE_0(K)  \to \wt{\EE}_{\textnormal{sm}}(k) \cong k^+$ is the reduction map:
\begin{align*}
\xymatrix{
\EE_0(K) \ar[rd]^{(x,y) \mapsto \frac{\overline{-x}}{\overline{y}}} \ar[dd]_{\Psi: (x,y) \mapsto -x/y} & \\
& k^+ \\
\widehat{\EE}(\OO_K) \ar[ru]_{x \mapsto \bar{x}} & .
}
\end{align*}
\end{enumerate}
\end{thm}
\begin{proof}
Let $\pi$ be a uniformizer for $\OO_K$. Consider the field extension $L = K(\rho)$ with $\rho^6 = \pi$. Then define the Weierstrass curve $\DD$ over $\OO_L$ by
$$
Y^2 + \alpha_1 XY + \alpha_3 Y = X^3 + \alpha_2 X^2 + \alpha_4 X + \alpha_6,
$$
where $\alpha_i = a_i / \rho^i$. There is a birational map $ \phi : \EE \times_{\OO_K} \OO_L \dashrightarrow \DD$, given by $\phi(X,Y)= (X/\rho^2,Y/\rho^3)$. The birational map $\phi$ induces an isomorphism on generic fibers, and hence a homeomorphism between $\EE(L)$ and $\DD(L)$. Using (\ref{defE1}) and the fact that we have $(x,y)\in \EE_0(L)$ if and only if $v_L(x),v_L(y)$ are both not greater than zero, one sees that $\phi$ induces a bijection $\EE_0(L) \stackrel{\sim}{\rightarrow} \DD_1(L)$, that all maps (\textit{a priori} just of sets) in the following diagram are well-defined, and that the diagram commutes:
\[
\xymatrix{ 
\EE_0(K) \ar[d]^{\Psi} \ar[r]^{\textnormal{incl}} &                  
\EE_0(L) \ar[d]^{\Psi_L}  \ar[r]^{\phi}       &            
\DD_1(L) \ar[d]^{\psi_L}                            \\\
\widehat{\EE}(\OO_K)  \ar[r]^{\textnormal{incl}} & \widehat{\EE}(\OO_L)  \ar[r]^{\cdot \rho} & \widehat{\DD}(\mm_L).
}
\]
Here the map $\Psi_{L} : \EE_0(L) \rightarrow \widehat{\EE}(\OO_L) $ is defined by $(x,y) \mapsto -x/y$, the rightmost lower horizontal arrow is multiplication by $\rho$, and the maps labeled $\,\textnormal{incl}\,$ are the obvious inclusions. Let $F_{\widehat{\DD}}$ be the formal group law of $\DD$. One calculates that
$$
\rho  F_{\widehat{\DD}}(X,Y) = F_{\widehat{\DD}}(\rho X, \rho Y).
$$ 
Hence $\cdot \rho$, $\psi_L$ (Proposition \ref{fund_isom}) and $\phi$ are homeomorphisms of groups. It follows that the group morphism $\Psi_L$ is a homeomorphism of groups. Hence $\Psi$ must be a homeomorphism onto its image. The map $\Psi_L$ is Galois-invariant, and hence by Galois theory, it follows that $\Psi$ is surjective, and that it is in fact a homeomorphism.

Now assume $6e(K/\Q_p) < p-1$, so that $v_L(p) = 6 v_K(p) = 6e(K/\Q_p)< p-1$. Now \cite[IV.6.4(b)]{SI} implies that $\EE_1(K)$ is isomorphic to $\mm_K$, and $\DD_1(L)$ to $\mm_L$. Since $\EE$ has additive reduction, we have $\wt{\EE}_{\sm}(k) \cong k^+$. We have an exact sequence
$$
0 \rightarrow \mm_K \rightarrow \EE_0(K) \rightarrow k^+ \rightarrow 0.
$$
In the diagram above, the topological group $\EE_0(K)$ is mapped homomorphically into the torsion-free group $\DD_1(L)$, hence it is itself torsion-free. It follows that $\EE_0(K)$ is isomorphic to $\Z_p^n$. This proves the second part.

The commutativity of the diagram follows directly.
\end{proof}

\subsection{$\EE_0(K)$: special cases}

In the previous subsection, we have seen that $\EE_0(K) \cong_{\Z_p} \Z_p^n$ if $p-1>6e$. We have $\EE_1(K) \cong_{\Z_p} \Z_p^n$ if $p-1>e$, or if $p=2$ and $e=1$ (Proposition \ref{144}). We will study what happens with $\EE_0(K)$ in the latter case, so assume $p-1>e$, or $p=2$ and $e=1$. We implicitly identify $\EE_i(K)$ with $ \widehat{\EE}(\mm_K^i)$. We have an exact sequence
\begin{align*}
0 \rightarrow \EE_1(K) \rightarrow \EE_0(K) \rightarrow \wt{\EE}_{\textnormal{sm}}(k) \cong k^+ \rightarrow 0.
\end{align*}
To compute the torsion of $\EE_0(K) $, we apply Lemma \ref{128}. We consider the induced group morphism $k \to \EE_1(K)/p \EE_1(K)= \EE_1(K)/\EE_{1+e}(K)$ (Proposition \ref{144}). Theorem \ref{tocomputewith} allows one to compute the map explicitly. Consider the power series $[p](T)$ corresponding to the curve. One obtains a map $g=[p](T): k \to \widehat{\EE}(\mm_K^1)/\widehat{\EE}(\mm_K^{1+e})$. In general, the torsion of $\EE_0(K)$ will be isomorphic to $\mathrm{ker}(g)$.  If $[p](T)= \sum_i b_i T^i$, one can show that only the terms $pT+\sum_{j>0:\ (pj-1)/6 \leq e} b_{pj} T^{pj}$ play a role (Lemma \ref{mult_by_p}). We see that the map is just $[p]T=pT$ if $(p-1)/6>e$, hence the kernel of $g$ is trivial, and we recover part of Theorem \ref{tocomputewith}. Given a field $K$, one can explicitly find conditions on the $a_i$ which determine the kernel of $g$ and hence the structure of $\EE_0(K)$. As an example, we study the case where $e=1$. 

\subsubsection{$e=1$}
This is the case where $K/\Q_p$ is unramified. We have a map $g: k \to \widehat{\EE}(\mm_K^1)/\widehat{\EE}(\mm_K^{2}) \cong k$. Since the latter isomorphism is very explicit, we do not need logarithm to understand it. The map $g$ is given by a polynomial. For $p>7$, this is the identity map and there is no torsion. 

For $p=3,5,7$ this polynomial is of the form $T-a T^p$. If $n=1$, then $\EE_0(K) \cong_{\Z_p} \Z_p \times \Z/p\Z$ if $a=1$ and $\Z_p$ otherwise. If $n>1$, then Lemma \ref{183} tells us that this equation has a non-trivial solution over $k$ if and only if $a$ is a $(p-1)$st power in $k$. In that case, $\EE_0(K)$ is isomorphic to $\Z_p^n \times \Z/p\Z$, and otherwise it is $\Z_p^n$. 

If $p=2$, the polynomial is of the form $T-aT^2-bT^4$. Lemma \ref{183} can be used to handle the case when $n=1,2$, but otherwise it is more tricky to see what the kernel of this map is. It shows that $\EE_0(K)$ is isomorphic to $\Z_2^n \times (\Z/2\Z)^i$ with $i \in \{0,1,2\}$. Here $2^i$ is the number of solutions of $T-aT^2-bT^4=0$ in $k$. 

For $p=2,3,5,7$ we explicitly compute the polynomial $g=[p](T)/p \pmod{p \OO_K}$.  \vspace{0.2cm}

\begin{tabular}{| c | c | }
\hline
  $p=2$ &  \\ \hline
  $[p](T)$ & $2T - a_1T^2 -2a_2T^3 + (a_1a_2-7a_3)T^4 + \ldots$. \\
  $g$ & $T-\overline{a_1/2} T^2-\overline{a_3/2}T^4$. \\ \hline
\end{tabular} \vspace{0.2cm}

\begin{tabular}{| c | c | }
\hline
  $p=3$ &  \\ \hline
  $[p](T)$ & $3T - 3a_1T^2 +(a_1^2-8a_2)T^3 + (12a_1a_2 - 39a_3)T^4 + \ldots$. \\
  $g$ & $T-\overline{8a_2/3} T^3$. \\ \hline
\end{tabular} \vspace{0.2cm}

\begin{tabular}{| c | c | }
\hline
  $p=5$ &  \\ \hline
  $[p](T)$ & $5T - 1248a_4 T^5 + \ldots$. \\
  $g$ & $T-\overline{3a_4/5} T^5. $. \\ \hline
\end{tabular} \vspace{0.2cm}

\begin{tabular}{| c | c | }
\hline
  $p=7$ &  \\ \hline
  $[p](T)$ & $7T - 6720a_4T^5 - 352944a_6 T^7 + \ldots$. \\
  $g$ & $T-\overline{4a_6/7} T^7$. \\ \hline
\end{tabular}  \vspace{0.2cm}

In all the above cases, if there is no torsion in $\EE_0(K)$, then the map $g$ is surjective, and by Lemma \ref{128} one has $p\EE_{0}(K)=\EE_1(K)$ (one identifies $\EE_0(K)$ with $\mathcal{O}_K$ after the isomorphism $\widehat{\EE}(\mm_K^1) \cong \m_K$). 

\begin{proof}[Proof of Theorem \ref{777} and Corollary \ref{555}]
Follows from the discussion above (and Lemma \ref{183}).
\end{proof}

We will now list some examples of elliptic curves over $\Q_p$ with additive reduction, such that their points of good reduction contain $p$-torsion points. In particular, all curves and torsion points are defined over $\Q$. We also have an example of a curve over $\Q_2(\zeta_3)$ having full $2$-torsion in $\Q(\zeta_3)$.

\begin{ex}\upshape
The elliptic curve
$$
E_2 : Y^2 + 2 Y = X^3 - 2
$$
has additive reduction at 2, and its 2-torsion point $(1,-1)$ is of good reduction. 

In fact, if one considers this curve over $\Q_2(\zeta_3)$, the unramified extension of $\Q_2$ of degree $2$, then the $2$ torsion contains $4$ points of good reduction defined over $\Q(\zeta_3)$:
\begin{align*}
\{\infty, (-\zeta_3,-1), (\zeta_3,-1), (1,-1) \}.
\end{align*}
\end{ex}

\begin{ex}\upshape
The elliptic curve
$$
E_3 : Y^2 = X^3 - 3X^2 + 3X
$$
has additive reduction at 3, and its 3-torsion point $(1,1)$ is of good reduction. 
\end{ex}

\begin{ex}\upshape
The elliptic curve
$$
E_5 : Y^2 - 5 Y = X^3 + 20X^2 - 15X
$$
has additive reduction at 5, and its 5-torsion point $(1,-1)$ is of good reduction. 
\end{ex}

\begin{ex}\upshape
The elliptic curve
$$
E_7 : Y^2 + 7XY  -28Y = X^3 + 7X- 35
$$
has additive reduction at 7, and its 7-torsion point $(2,1)$ is of good reduction. 
\end{ex}

\begin{ex}\upshape
Consider the curve $E_8: Y^2= X(X-2)(X-4)$ over $\Q$. This elliptic curve has additive reduction at $2$. One has $E_8(\Q)[2^{\infty}]=\{(0,0),(2,0), (4,0),\infty\}$, and this group is isomorphic to $V_4$. All torsion points, except $\infty$, map to the singular point of the reduction. In fact, over $\Q_2$, the points of good reduction are isomorphic to the group $\Z_2$. 
\end{ex}

Sometimes, we can use Theorem \ref{777} to show that rational points have infinite order.

\begin{ex}\upshape
Consider the curve $E_9\colon Y^2=X^3-2$ over $\Q$. This curve has additive reduction at $2$ and the point $(3,5)$ has good reduction. By Theorem \ref{777} the point has infinite order. Similarly, one can consider $E_{10}\colon Y^2=X^3+3$, which has additive reduction at $3$. The point $(1,2)$ has good reduction and infinite order by Theorem \ref{777}.
\end{ex}

\subsubsection{$e>1$}

One can generalize the results above slightly. Starting with a torsion free $\EE_i(K)$ such that $p \EE_i(K)=\EE_{i+e}(K)$, one can use Lemma \ref{128} to check if there is no torsion in $\EE_{i-1}(K)$. One of the problems in our construction, is that in general, even if $\EE_i(K)$ is torsion free, it is not true that $p \EE_{i}(K)=\EE_{i+e}(K)$, and hence our methods do not work directly. To do computations in this case, one needs to really compute the logarithm maps giving the isomorphism with $\m_K^i$. We will give a brief example in which we study a ramified case.

\begin{ex}\upshape
Take $K=\Q_2(\sqrt{2})$. One has $k=\F_2$ and $\m_K=\sqrt{2} \Z_2[\sqrt{2}]$. One has an exact sequence $0 \to \EE_1(K) \to \EE_0(K) \to k \to 0$. One has $\EE_i(K) \cong \m^i$ for $i \geq 1$ (Proposition \ref{144}). A computation gives $[2](T) = 2T-a_1T^2+(a_1a_2-7a_3)T^4+\ldots$. This polynomial induces a map $g: k \to \widehat{\EE}(\mm_K^1)/\widehat{\EE}(\mm_K^3)$, and the latter group is isomorphic to $\m_K/\m_K^3$ by a logarithm map. The map $k \to \m_K/\m_K^3$ is of the form $s \mapsto (\sqrt{2}c_0+2 c_1)s$ with $c_0, c_1 \in \{0,1\}$. Unfortunately, one needs the logarithm map explicitly to see which map this is in terms of the $a_i$. One then uses Lemma \ref{128} to describe the torsion of $\EE_0(K)$. If there is no torsion, one can descrite $\EE_0(K)$ as a subgroup of $K$.

\begin{align*}
\begin{tabular}{| c  | c | c | c | c | c| }
  \hline			
  $c_0$ & $c_1$ & $\mathrm{im}(g)$ & $\mathrm{ker}(g)$ & $\EE_0(K)$ & $\EE_0(K)[2]$ \\ \hline
$0$ & $0$ & $\{0\}=\m_K^3/\m_K^3 $ & $k$ & -- & $\Z/2\Z$ \\
$0$ & $1$ & $\{0,2\}=\m_K^2/\m_K^3$ & $0$ & $\mathcal{O}_K$ & $0$ \\
$1$ & $0$ & $\{0,\sqrt{2}\}$ & $0$ & $1\sqrt{2} \Z_p \oplus 2 \Z_p$ & $0$ \\
$1$ & $1$ & $\{0,\sqrt{2}+2\}$ & $0$ & $(1/\sqrt{2}+1)\Z_p \oplus 2 \Z_p$ & $0$ \\
\hline
\end{tabular}
\end{align*}
Without explicit computations of the logarithm, one cannot distinguish between the third and fourth option. In the cases $(c_0,c_1)=(1,0), (1,1)$, $\EE_0(K)$ is not identified with $\mathcal{O}_K$ and hence $p \EE_0(K) \neq \EE_2(K)$. 
\end{ex}

\section{Acknowledgements}
It is a pleasure to thank Ronald van Luijk and Sir Peter Swinnerton-Dyer for many useful remarks.

\end{document}